\documentclass[12pt,a4paper]{amsart}
\usepackage{a4wide}
\usepackage{amsmath, amssymb, amsfonts,enumerate}
\usepackage[all]{xy}
\usepackage{amscd}
\usepackage{comment}
\usepackage{mathtools}
\usepackage{tikz} 
\usepackage{hyperref}
\newtheorem{theorem}{Theorem}[section]

\newtheorem{proposition}[theorem]{Proposition}

 \theoremstyle{definition}

 \newtheorem{example}[theorem]{Example}

\numberwithin{equation}{section}
\newcommand {\Z}{\mathbb{Z}} 

\newcommand{\OO}{\mathcal{O}}

\DeclareMathOperator{\diam}{diam}

\DeclareMathOperator{\htop}{h_{\scriptstyle  \text{\rm top}}}

\begin{document}
\title[A note on surjective endomorphisms]{A note on surjective endomorphisms of Smale spaces}
\author{Tullio Ceccherini-Silberstein}
\address{Dipartimento di Ingegneria, Universit\`a del Sannio, I-82100 Benevento, Italy}
\address{Istituto Nazionale di Alta Matematica ``Francesco Severi'', I-00185 Rome, Italy}
\email{tullio.cs@sbai.uniroma1.it}
\author{Michel Coornaert}
\address{Universit\'e de Strasbourg, CNRS, IRMA UMR 7501, F-67000 Strasbourg, France}
\email{michel.coornaert@math.unistra.fr}
\begin{abstract}
We give  conditions that are all equivalent to surjectivity for endomorphisms of  irreducible Smale spaces.
This extends results previously obtained by Hedlund~\cite{hedlund-endomorphisms}  for full shifts, 
by Coven and Paul~\cite{coven-paul} and by Lind and Marcus~\cite{lind-marcus} for irreducible subshifts of finite type.
\end{abstract}
\date{\today}
\maketitle

\tableofcontents

\section{Introduction}

A \emph{dynamical system} is a pair $(X,f)$, where $X$ is a compact metrizable space, called the \emph{phase space}, 
and $f \colon X \to X$ is a homeomorphism.
An \emph{endomorphism} of a dynamical system $(X,f)$ is a continuous map $\tau \colon X \to X$ which commutes with $f$, i.e., such that $\tau \circ f = f \circ \tau$.
One says that a dynamical system $(X,f)$ is \emph{irreducible} (or \emph{positively topologically transitive})
if there exists a point $x \in X$ whose forward orbit $\OO^+(x) \coloneqq \{f^n(x) : n = 1,2,\dots\}$ of $x$ is dense in $X$.
Two points $x,y$ in the phase space of a dynamical system $(X,f)$ are called \emph{homoclinic} if the points $f^n(x)$ and $f^n(y)$ become arbitrarily close as $|n| \to \infty$
(i.e., $d(f^n(x),f^n(y)) \to 0$ as $|n| \to \infty$, where $d$ is any compatible metric on $X$).
Homoclinicity is an equivalence relation on the phase space of a dynamical system.
One says that an endomorphism of a dynamical system is \emph{pre-injective} if its restriction to each homoclinicity class is injective. 
\par
Roughly speaking, a \emph{Smale space} is a dynamical system $(X,f)$ such that  $X$ admits a local product structure 
for which   $f$ is uniformly contracting in one direction and uniformly expanding in the other
(see Section~\ref{sec:smale} for a precise definition).
The class of Smale spaces was introduced by Ruelle in~\cite{ruelle-thermodynamic-2nd}. 
It includes in particular 
Anosov diffeomorphisms~\cite{anosov-geodesic}, 
non-wandering sets and  basic sets of Smale's Axiom A diffeomorphisms~\cite{smale-bams}, 
subshifts of finite type~\cite{lind-marcus-second},
expansive principal algebraic dynamical systems~\cite{schmidt-book}, \cite{li-goe-spec-2019}, \cite{meyerovitch-2019},
warious types of solenoids and  attractors~\cite{williams_1974}, \cite{wieler_2014}, 
as well as dynamical systems constructed by Nekrashevych from certain self-similar groups~\cite{nekrashevych}. 
\par
The phase space of an irreducible Smale space admits a prefered invariant Borel probability measure,
called the \emph{Bowen measure} of the Smale space.
It is characterized by the fact that it is the unique invariant Borel probability measure which maximizes measure-theoretic entropy.
\par
Our main result is the following.

\begin{theorem}
\label{t:main}
Let $(X,f)$ be an irreducible Smale space 
and let $\mu$ denote the Bowen probability measure on $X$.
Suppose that $\tau \colon X \to X$ is an endomorphism of $(X,f)$.
Then the following conditions are equivalent:
\begin{enumerate}[\rm (i)]
\item
$\tau$ is surjective;
\item
$\tau$ preserves the measure $\mu$;
\item
$\mu(\tau(X)) = 1$;
\item
$\htop(\tau(X),f) = \htop(X,f)$ (here $\htop(\cdot,\cdot)$ denotes topological entropy);
\item
$\tau$ is pre-injective;
\item
$\tau$ is bounded-to-one;
\item 
$\tau$ is finite-to-one;
\item
$\tau$ is countable-to-one.
 \end{enumerate}
\end{theorem}

All known examples of compact manifolds admitting Anosov diffeomorphisms are infra-nilmanifolds.
It is an open question  whether every Anosov diffeomorphism on a compact manifold is irreducible.
However, Manning~\cite{manning-1974} proved that any Anosov diffeomorphism on a infra-nilmanifold is irreducible.
It follows that all conditions in Theorem~\ref{t:main} are equivalent for endomorphisms of Anosov diffeomorphisms of infra-nilmanifolds.
\par
Theorem~\ref{t:main} extends results obtained by 
Hedlund~\cite{hedlund-endomorphisms} in the case of endomorphisms of full shifts, 
by Coven and Paul~\cite{coven-paul} and by Lind and Marcus~\cite[Theorem~8.1.16]{lind-marcus}  in the case of endomorphisms of irreducible subshifts of finite type.
\par
The paper is organized as follows.
In Section~\ref{preliminaries}, we fix terminology and collect some auxiliary results in the theory of dynamical systems. 
The proof of Theorem~\ref{t:main} is given in Section~\ref{proof}. 
In Section~\ref{counter-examples}, we present some counter-examples showing that Theorem~\ref{t:main} may become false if the hypotheses are weakened. 

\section{Preliminaries}
\label{preliminaries}

In this section, we introduce  some terminology 
and collect  auxiliary results that will be used in our proof of Theorem~\ref{t:main}.

\subsection{Sets and maps}
We write $\Z$ for the group of integers.
\par
We write $|X|$ for the cardinality of a set $X$.
\par
One says that a set $X$ is countable if $|X| \leq |\Z|$, i.e., $X$ is either finite or countably infinite.
\par
Let $f \colon X \to Y$ be a map between sets.
The \emph{fiber} of $f$ at $y \in Y$ is the pre-image set $f^{-1}(y) \coloneqq \{x \in X : f(x) = y\}$.
One says that  $f$ is \emph{finite-to-one} (resp.~\emph{countable-to-one}) if 
$f^{-1}(y)$ is finite (resp.~countable) for every $y \in Y$.
One says that  $f$ is \emph{bounded-to-one} if there exists an integer $n \geq 0$ such that
$|f^{-1}(y)| \leq n$ for all $y \in Y$.

\subsection{Entropy of dynamical systems}
Recall that a  \emph{dynamical system} is a pair $(X,f)$, where $X$ is a compact metrizable space and $f \colon X \to X$ is a homeomorphism.
We shall often write $X$ instead of $(X,f)$ if the homeomorphism $f$ is clear from the context.
\par
One says that a dynamical system $(Y,g)$ is a \emph{subsystem} of the dynamical system $(X,g)$ if $Y$ is a closed $f$-invariant subset of $X$ and $g = f|_Y \colon Y \to Y$, where $f|_Y$  is the restriction of $f$ to $Y$.
To simplify, we shall then write $f$ instead of $f|_Y$. 
\par
Let $(X,f)$ be a dynamical system.
We  write $\htop(X,f)$ for the \emph{topological entropy} of $(X,f)$ (see \cite[Section~2.5]{brin-stuck}). 
Let $M(X,f)$ denote the set of all $f$-invariant Borel probability measures on $X$.
Given $\mu \in M(X,f)$, we write $h_\mu(X,f)$ for the \emph{measure-theoretic entropy} of $(X,f)$ with respect to $\mu$
(see \cite[Section~9.3]{brin-stuck}).
By the \emph{variational principle} \cite[Theorem~9.5.4]{brin-stuck}, one has
\[
\htop(X,f) = \sup_{\mu \in M(X,f)} h_\mu(X,f).
\]
One says that $\mu \in M(X,f)$ has \emph{maximal entropy} if $\htop(X,f) = h_\mu(X,f)$.  
\par
Let $(X,f)$ and $(Y,g)$ be dynamical systems.		
One says that a map $\pi \colon X \to Y$ is a \emph{factor map} if $\pi$ is continuous, surjective, and $\pi \circ f = g \circ \pi$.
If there exists a factor map $\pi \colon X \to Y$, then one has $\htop(X,f) \geq \htop(Y,g)$ \cite[Proposition~2.5.6]{brin-stuck}.
By the following result, this inequality becomes an equality under the additional hypothesis that all fibers of $\pi$ are countable.

\begin{proposition}
\label{p:ent-countable-to-one-factor}
Let $(X,f)$ and $(Y,g)$ be dynamical systems.
Suppose that there exists a countable-to-one factor map $\pi \colon X \to Y$.
Then one has $\htop(X,f) = \htop(Y,g)$.
\end{proposition}

\begin{proof}
This is a particular case of the second assertion in Theorem~5.7 of~\cite{yan-conditional-entropy} since $\Z$ is an infinite amenable group.
\end{proof}

\subsection{Intrinsic ergodicity}
Following Weiss~\cite{weiss-intrinsically-ergodic}, one says that a dynamical system $(X,f)$ is \emph{intrinsically ergodic} if $\htop(X,f) < \infty$ and there exists a unique measure $\mu \in M(X,f)$ with maximal entropy. 
One says that an intrinsically ergodic dynamical system $(X,f)$ is \emph{with full support} if the support
of the measure  with maximal entropy $\mu$ of $(X,f)$ is the whole space $X$, i.e.,
$\mu(U) > 0$ for every non-empty open subset $U$ of $X$.
\par
From the results in Section~2 of~\cite{coven-paul}, we get the following.
  
\begin{proposition}
\label{p:main_1}
Let $(X,f)$ be an intrinsically ergodic dynamical system with full support and let  
$\mu$ denote the measure with maximal entropy of $(X,f)$.
Suppose that $\tau \colon X \to X$ is an endomorphism of $(X,f)$.
Then the following conditions are equivalent:
\begin{enumerate}[\rm (i)]
\item
$\tau$ is surjective;
\item
$\tau$ preserves $\mu$ (that is, $\mu(\tau^{-1}(B)) = \mu(B)$ for every Borel subset $B \subset X$);\item
$\mu(\tau(X)) = 1$.
\end{enumerate}
\end{proposition}

\begin{proof}
The equivalence between (i) and (ii) is given by Corollary~2.3 in \cite{coven-paul}.
\par
If $\tau$ is surjective then  $\mu(\tau(X)) = \mu(X) = 1$.
Conversely, if $\mu(\tau(X)) = 1$ then $\tau(X) =X$ since $\tau(X)$ is closed in $X$ and the support of $\mu$ is $X$.
This shows the equivalence between (i) and (iii).
\end{proof}

\subsection{Expansivity}
Let $(X,f)$ be a dynamical system and let $d$ be a compatible metric on $X$.
One says that the dynamical system $(X,f)$ is \emph{expansive}
if there exists $c > 0$ such that, for all distinct $x,y \in X$, there exists $n \in \Z$ satisfying $d(f^n(x),f^n(y)) \geq c$.
Such a constant $c$ is then called an \emph{expansivity constant} for $(X,d,f)$.
The fact that $(X,f)$ is expansive or not  does not depend on the choice of the compatible metric $d$ by compactness of $X$.
Clearly, every subsystem of an expansive dynamical system is itself expansive.
\par
It is known that if $(X,f)$ is an expansive dynamical system then there exists $\mu \in M(X,f)$ with maximal entropy
(see e.g.~\cite[Exercise~9.5.2]{brin-stuck}). Note that such a $\mu$ is not unique in general. 

\begin{proposition}
\label{p:main-2}
Let $(X,f)$ be an expansive intrinsically ergodic dynamical system with full support. 
Suppose that $\tau \colon X \to X$ is an endomorphism of $(X,f)$.
Then the following conditions are equivalent:
\begin{enumerate}[\rm (i)]
\item
$\tau$ is surjective;
\item
$\htop(\tau(X),f) = \htop(X,f)$.
\end{enumerate}
\end{proposition}

\begin{proof}
The implication (i) $\implies$ (ii) is trivial.
\par
Conversely, suppose (ii).
As the dynamical system $(\tau(X),f)$ is expansive, there exists $\nu \in M(\tau(X),f)$ with maximal entropy.
Consider the measure $\mu \in M(X,f)$ defined by $\mu(B) = \nu(B \cap \tau(X))$ for every Borel subset $B \subset X$.
Then $h_{\mu}(X,f) = h_{\nu}(\tau(X),f) = \htop(\tau(X),f) = \htop(X,f)$.
It follows that $\mu$ is the measure with maximal entropy of $(X,f)$.
Consequently, the support of $\mu$ is $X$.
Since the support of $\mu$ is contained in $\tau(X)$ by definition, we conclude that $\tau(X) = X$.
This shows that (ii) implies (i).
\end{proof}

\subsection{The pseudo-orbit tracing property}
Let $(X,f)$ be a dynamical system and let $d$ be a compatible metric on $X$.
\par
Let $\varepsilon > 0$. 
One says that a sequence $(p_n)_{n \in \Z}$ of points of $X$ is $\varepsilon$-\emph{traced} by a point $z \in X$ if one has $d(f^n(z), p_n) \leq \varepsilon$ for all $n \in \Z$.
\par
Let $\delta > 0$.
One says that a sequence $(p_n)_{n \in \Z}$ of points of $X$ is a $\delta$-\emph{pseudo-orbit} if one has
$d(f(p_n),p_{n+1}) \leq \delta$ for all $n \in \Z$.
\par
One says that the dynamical system $(X,f)$ has the \emph{pseudo-orbit tracing property} (or \emph{shadowing property}) if for every $\varepsilon > 0$, there exists $\delta > 0$ such that every $\delta$-pseudo-orbit in $X$ is $\varepsilon$-traced by some point in $X$.
By compactness of $X$,  the fact that $f$ has the pseudo-orbit tracing property or not does not depend on the choice of the compatible metric $d$.

\begin{proposition}
\label{p:pre-injective-b-to-one}
Let $(X,f)$ be an expansive dynamical system with the pseudo-orbit tracing property.
Let $\tau \colon X \to X$ be a pre-injective endomorphism of $(X,f)$.
Then $\tau$ is bounded-to-one.
\end{proposition}

\begin{proof}
Fix a compatible metric $d$ on $X$  and let $c > 0$ be an expansivity constant for $(X,d,f)$.
As $\tau$ is uniformly continuous by compactness of $X$, there exists $0 <\varepsilon < c$ such that
$d(\tau(x),\tau(y)) < c$ whenever $x,y \in X$ satisfy $d(x,y) < \varepsilon$.
\par
By the pseudo-orbit tracing property, there exists $\delta>0$ such that every $\delta$-pseudo-orbit in $X$  is $\varepsilon$-traced by some point of $X$.
\par 
We claim that there exists  a finite cover $U_1,\ldots,U_m$ of $X$ such that $\diam(U_i) \coloneqq \sup\{d(x,y) : x,y \in U_i\} \leq \delta$ for each $i=1,\ldots,m$. 
Indeed, the open balls   $B_{\delta/2}(x)$ of radius $\delta/2$, $x \in X$, cover $X$. 
By compactness of $X$, there exist finitely many  points  $x_1, \ldots, x_m \in X$ such that
 the open balls $B_{\delta/2}(x_1), \dots,   B_{\delta/2}(x_m)$ cover $X$. 
We can then take  $U_i \coloneqq B_{\delta/2}(x_i)$, $i=1,\ldots, m$.
\par
Let us show now that every fiber of $\tau$ has cardinality at most $m^2$.
Suppose by contradiction that there exists $y \in X$ such that $| \tau^{-1}(y) | \geq m^2+1$.
Pick distinct points $x_1,\ldots,x_{m^2+1}$  in $\tau^{-1}(y)$.
\par
By expansivity, for all distinct $ i, j \in \{1,\dots,m^2+1\}$, there exists $k_{ij} \in \Z$ such that
\begin{equation}
\label{e:separation}
d(f^{k_{ij}}(x_i),f^{k_{ij}}(x_j)) \geq c.
\end{equation}
Set
\[ 
N \coloneqq 1+ \max\{|k_{ij}|: i,j \in \{1,\dots,m^2+1\} \text{ and } i \not= j\}.
\] 
\par
For each $i \in \{1,\ldots, m^2+1\}$, choose
$a_i,b_i \in \{1, \ldots, m\}$ such that  $f^{-N}(x_i) \in U_{a_i}$ and  $f^N(x_i) \in U_{b_i}$.
By the pigeon-hole principle, we can find distinct $i ,j \in \{1,\dots,m^2 + 1\}$ such that $(a_i,b_i) = (a_j,b_j)$.
We then have
\begin{equation}
\label{e:inegalites-aux-bords-de-l'intervalle}
d(f^{-N}(x_i),f^{-N}(x_j)) \leq \delta \text{ and } d(f^N(x_i),f^N(x_j)) \leq  \delta.
\end{equation}
For every $n \in \Z$, define the point  $p_n \in X$ by setting
\[
p_n \coloneqq 
\begin{cases}
f^n(x_i) & \text{ if } n < -N \text{ or } n \geq N \\
f^n(x_j) & \text{ if } -N \leq n < N. 
\end{cases}
\]
We then have $f(p_n) = p_{n+1}$ and hence $d(f(p_n),p_{n + 1}) = 0$ for all $n \in \Z \setminus \{-N -1,N - 1\}$.
On the other hand, we have 
$f(p_{-N-1}) = f(f^{-N-1}(x_i)) = f^{-N}(x_i)$, $p_{-N} = f^{-N}(x_j)$, 
$f(p_{N-1}) =f(f^{N-1}(x_j)) = f^N(x_j)$, $p_N = f^N(x_i)$. 
We deduce that   $d(f(p_n),p_{n + 1}) \leq \delta$ for each 
$n \in \{-N -1,N - 1\}$
by using~\eqref{e:inegalites-aux-bords-de-l'intervalle}. 
Thus, the sequence  $(p_n)_{n\in \Z}$ is a $\delta$-pseudo-orbit. 
By our choice of $\delta$,   there exists a point $z \in X$ which $\varepsilon$-traces the sequence $(p_n)_{n \in \Z}$, i.e., such that
\begin{equation}
\label{e:tracing}
d(f^n(z),p_n) \leq \varepsilon \text{ for all } n \in \Z.
\end{equation}
This implies in particular
\begin{equation}
\label{e:tracing-x-i}
d(f^n(z),f^n(x_i)) \leq \varepsilon \text{ for all } n \in \Z \text{ such that } n < -N \text{ or } n \geq N. 
\end{equation}
On the other hand, since $\tau(x_i) = \tau(x_j)=y$ and $\tau$ commutes with $f$, we have
$\tau(p_n)=f^n(y)$ for all $n \in \Z$.
By our choice of $\varepsilon$, we deduce that
$d(f^n(\tau(z)),f^n(y)) = d(\tau(f^n(z)),\tau(p_n)) < c$ for all
$n \in \Z$. 
As $c$ is an expansivity constant, this implies that  $\tau(z) = y = \tau(x_i)$.
\par
We next show that the points $z$ and $x_i$ are homoclinic, that is, $d(f^n(z),f^n(x_i))\to 0$ as $|n|\to\infty$.
By the definition of $p_n$ and \eqref{e:tracing}, we have
\begin{equation} \label{e:tail-bound}
d(f^n(z),f^n(x_i)) = d(f^n(z),p_n) \leq \varepsilon \text{ whenever  }|n|>N.
\end{equation}
Suppose that $z$ and $x_i$ are not homoclinic.
Then there exist $\eta>0$ and a sequence $(n_\ell)_{\ell\geq 1}$ in $\Z$ such that $|n_\ell|\to\infty$ and $d(f^{n_\ell}(z),f^{n_\ell}(x_i))\geq\eta$ for every $\ell\geq 1$.
\par
By compactness of $X\times X$, after passing to a subsequence if necessary,  we may assume that $f^{n_\ell}(z)\to u$ and $f^{n_\ell}(x_i)\to v$ for some $u,v\in X$.
By continuity of the metric, $d(u,v)\geq\eta$, so $u\neq v$.
\par
 Fix $k\in\Z$.
Since $|n_\ell+k|\to\infty$, we deduce from \eqref{e:tracing-x-i} that 
$d(f^k(f^{n_\ell}(z)),f^k(f^{n_\ell}(x_i)) = d(f^{n_\ell+k}(z),f^{n_\ell+k}(x_i)) \leq \varepsilon$ for all sufficiently large $\ell$.
By continuity of $f^k$, passing to the limit as $\ell \to \infty$, gives us $d(f^k(u),f^k(v))\leq\varepsilon<c$. 
As this holds for all $k\in\Z$ and $c$ is an expansivity constant, this implies $u=v$, a contradiction. 
Thus $z$ and $x_i$ are homoclinic.
\par
Since $\tau(z) = y = \tau(x_i)$, we obtain $z=x_i$ by pre-injectivity.
Consequently, \eqref{e:tracing} gives us
$d(f^n(x_i),f^n(x_j)) = d(f^n(z),p_n)  \leq \varepsilon < c$ whenever  $-N \leq n < N$.
As $-N < k_{ij} < N$ by definition of $N$, this contradicts \eqref{e:separation}.
 Therefore every fiber of $\tau$ has cardinality at most $m^2$.
This shows that  $\tau$ is bounded-to-one.
\end{proof}

\subsection{Smale spaces}
\label{sec:smale}
Ruelle's original definition of a Smale space (cf.~\cite[Section~7.1]{ruelle-thermodynamic-2nd}, \cite[Definition~2.1.6]{putnam-homology-smale-spaces-2014},
\cite[Section~2.2]{putnam-notes-smale-spaces-2015}) goes as follows.
One says that a dynamical system $(X,f)$ is a \emph{Smale space} if there exist a compatible metric $d$ on $X$,  constants $\varepsilon > 0$ and $0 < \lambda < 1$,
and a continuous map 
\[
 [\cdot,\cdot] \colon \Delta_{\varepsilon} \coloneqq \{(x,y) \in X \times X : d(x,y) \leq \varepsilon\} \to X
\]
satisfying the following conditions for all $x,y,z \in X$:
\begin{enumerate}[\rm (Sm1)]
\item
$[x,x] = x$ for all $x \in X$;
\item
$[[x, y], z] = [x, z]$ whenever both sides are defined;
\item
$[x, [y, z]] = [x, z]$ whenever both sides are defined;
\item
$[f(x),f(y)] = f([x,y])$ whenever both sides are defined;
\item
$d(f(y),f(z)) \leq \lambda d(y,z)$ whenever $(y,x),(z,x) \in \Delta_\varepsilon$ and $[y,x] = x= [z,x]$;
\item
$d(f^{-1}(y),f^{-1}(z)) \leq \lambda d(y,z)$ whenever $(x,y),(x,z) \in \Delta_{\varepsilon}$ and $[x,y] = x = [x,z]$.
\end{enumerate}
Every Smale space is expansive \cite[Section~7.3]{ruelle-thermodynamic-2nd}, \cite[Proposition~2.1.9]{putnam-homology-smale-spaces-2014}
and has the pseudo-orbit tracing property 
\cite[Section~7.3]{ruelle-thermodynamic-2nd}, \cite[Theorem~4.3.3]{putnam-notes-smale-spaces-2015}.
Conversely, it is known that every expansive dynamical system having  the pseudo-orbit tracing property is a Smale space
\cite[Theorem]{ombach-smale-spaces}.
\par
Every irreducible Smale space is intrinsically ergodic with full support 
(cf.~\cite[Section~7]{ruelle-thermodynamic-2nd}, \cite{putnam-bowen-measures}). 
The unique measure with maximal entropy of a Smale space is called its \emph{Bowen measure}.
\par
Given a finite set $A$, a subshift $X \subset A^{\Z}$ is a Smale space if and only if $X$ is of finite type
(cf.~\cite{csc-goe-smale} and the references therein).
The Bowen measure of an irreducible subshift of finite type is also called the \emph{Parry measure} of the subshift
(see~\cite{parry-intrinsic}, \cite[Section~13.3]{lind-marcus}).

\section{Proof of the main result}
\label{proof}

The equivalences   (i) $\iff$ (ii)$\iff$ (iii)  follow from Proposition~\ref{p:main_1}.
\par
The equivalence (i) $\iff$ (iv) follows from Proposition~\ref{p:main-2}.
\par
The equivalence  (i) $\iff$ (v) is the Garden of Eden theorem for irreducible Smale spaces established in~\cite[Theorem~1.1]{csc-goe-smale}. 
\par
We clearly have  (vi) $\implies$ (vii) $\implies$ (viii).
\par
Implication (viii) $\implies$ (iv) immediately follows from Proposition~\ref{p:ent-countable-to-one-factor}.
\par
To conclude, it suffices to prove that (v) implies (vi).
This follows from Proposition~\ref{p:pre-injective-b-to-one}.

\section{Some counterexamples}
\label{counter-examples}

The following example shows that a Smale space can admit a pre-injective endomorphism which is not surjective (cf.~\cite[Example~8.3]{ccs}).

\begin{example}
Let $A \coloneqq \{0,1\}$. Let $X \subset A^{\Z}$ denote the subshift of finite type 
admitting $\{01,10\}$ as a defining set of forbidden words.
Thus, $X$ consists of the two constant configurations $0^{\Z}$ and $1^{\Z}$.
The map  $\tau \colon X \to X$ defined by $\tau(0^{\Z}) = \tau(1^{\Z}) = 0^{\Z}$ is 
an endomorphism of $X$ which is pre-injective but not surjective.
Observe that $\tau$ is bounded-to-one since $|\tau^{-1}(y)| \leq 2$ for every $y \in X$.
\end{example}

The following example shows that a Smale space can admit a surjective endomorphism which is not pre-injective (cf.~\cite[Example~8.7]{ccs}). 

\begin{example}
Let $A \coloneqq \{0,1,2\}$. Let $X \subset A^{\Z}$ denote the subshift of finite type admitting
$\{10,20,21\}$ as a defining set of forbidden words. 
Consider the map $\tau \colon X \to X$, where $\tau(x) \in X$ is obtained from $x \in X$  by replacing the word $112$ (if it appears in $x$) by the word $122$.
Clearly,  $\tau$ is a surjective endomorphism of $X$.
On the other hand,  $x_0 \coloneqq \dots00.1222\dots$ and $x_1 \coloneqq \dots00.1122\dots$ are distinct homoclinic configurations in $X$ with the same image $\tau(x_0) = \tau(x_1) = x_0$,
so that $\tau$ is not pre-injective.
Observe  that $\tau$ is bounded-to-one since $|\tau^{-1}(y)| \leq 2$ for every $y \in X$.  
\end{example}

The following example shows that an irreducible sofic subshift can admit a surjective endomorphism which is neither pre-injective nor countable-to-one.
Note that every irreducible sofic subshift is expansive and intrinsically ergodic with full support (cf.~\cite{weiss-sofic-1972}, \cite[Section~13.3]{lind-marcus}).

\begin{example}
Let $A \coloneqq \{0,1\}$. Let $X \subset A^{\Z}$ denote the \emph{even subshift}, i.e., the subshift
admitting $\{01^{2n + 1}0 : n = 0,1,\dots\}$ as a defining set of forbidden words.
It is known that $X$ is an irreducible sofic subshift (see e.g.~\cite[Example~2.1.9 and Section~1.3]{lind-marcus-second}).
Fiorenzi~\cite{fiorenzi-sofic} proved that
$X$ admits a surjective endomorphism which is not pre-injective (see also~\cite[Exercise~6.74]{csc-ecag}).
Such an endomorphism is not countable-to-one by~\cite[Theorem~8.1.19]{lind-marcus-second}.
On the other hand, Fiorenzi~\cite[Corollary~2.21]{fiorenzi-sofic} (see also \cite[Exercise~6.111]{csc-ecag})
proved that every pre-injective endomorphism of an irreducible sofic subshift is surjective.
\end{example}


\begin{thebibliography}{10}

\bibitem{anosov-geodesic}
{\sc D.~V. Anosov}, {\em Geodesic flows on closed {R}iemann manifolds with
  negative curvature}, vol.~No. 90 (1967) of Proceedings of the Steklov
  Institute of Mathematics, American Mathematical Society, Providence, RI,
  1969.
\newblock Translated from the Russian by S. Feder.

\bibitem{brin-stuck}
{\sc M.~Brin and G.~Stuck}, {\em Introduction to dynamical systems}, Cambridge
  University Press, Cambridge, 2002.

\bibitem{csc-ecag}
{\sc T.~Ceccherini-Silberstein and M.~Coornaert}, {\em Exercises in cellular
  automata and groups}, Springer Monographs in Mathematics, Springer, Cham,
  [2023] \copyright 2023.
\newblock With a foreword by Rostislav I. Grigorchuk.

\bibitem{csc-goe-smale}
\leavevmode\vrule height 2pt depth -1.6pt width 23pt, {\em A garden of {E}den
  theorem for {S}male spaces}, Mosc. Math. J., 25 (2025), pp.~479--494.

\bibitem{ccs}
{\sc T.~Ceccherini-Silberstein, M.~Coornaert, and V.~Salo}, {\em On surjunctive
  and injunctive subshifts of finite type}, arXiv:2606.17757.,  (2026).

\bibitem{coven-paul}
{\sc E.~M. Coven and M.~E. Paul}, {\em Endomorphisms of irreducible subshifts
  of finite type}, Math. Systems Theory, 8 (1974/75), pp.~167--175.

\bibitem{fiorenzi-sofic}
{\sc F.~Fiorenzi}, {\em The {G}arden of {E}den theorem for sofic shifts}, Pure
  Math. Appl., 11 (2000), pp.~471--484.

\bibitem{hedlund-endomorphisms}
{\sc G.~A. Hedlund}, {\em Endomorphisms and automorphisms of the shift
  dynamical system}, Math. Systems Theory, 3 (1969), pp.~320--375.

\bibitem{putnam-bowen-measures}
{\sc D.~B. Killough and I.~F. Putnam}, {\em Bowen measure from heteroclinic
  points}, Canad. J. Math., 64 (2012), pp.~1341--1358.

\bibitem{li-goe-spec-2019}
{\sc H.~Li}, {\em Garden of {E}den and specification}, Ergodic Theory Dynam.
  Systems, 39 (2019), pp.~3075--3088.

\bibitem{lind-marcus}
{\sc D.~Lind and B.~Marcus}, {\em An introduction to symbolic dynamics and
  coding}, Cambridge University Press, Cambridge, 1995.

\bibitem{lind-marcus-second}
\leavevmode\vrule height 2pt depth -1.6pt width 23pt, {\em An introduction to
  symbolic dynamics and coding}, Cambridge Mathematical Library, Cambridge
  University Press, Cambridge, second~ed., 2021.

\bibitem{manning-1974}
{\sc A.~Manning}, {\em There are no new {A}nosov diffeomorphisms on tori},
  Amer. J. Math., 96 (1974), pp.~422--429.

\bibitem{meyerovitch-2019}
{\sc T.~Meyerovitch}, {\em Pseudo-orbit tracing and algebraic actions of
  countable amenable groups}, Ergodic Theory Dynam. Systems, 39 (2019),
  pp.~2570--2591.

\bibitem{nekrashevych}
{\sc V.~Nekrashevych}, {\em Groups and topological dynamics}, vol.~223 of
  Graduate Studies in Mathematics, American Mathematical Society, Providence,
  RI, [2022] \copyright 2022.

\bibitem{ombach-smale-spaces}
{\sc J.~Ombach}, {\em Equivalent conditions for hyperbolic coordinates},
  Topology Appl., 23 (1986), pp.~87--90.

\bibitem{parry-intrinsic}
{\sc W.~Parry}, {\em Intrinsic {M}arkov chains}, Trans. Amer. Math. Soc., 112
  (1964), pp.~55--66.

\bibitem{putnam-homology-smale-spaces-2014}
{\sc I.~F. Putnam}, {\em A homology theory for {S}male spaces}, Mem. Amer.
  Math. Soc., 232 (2014), pp.~viii+122.

\bibitem{putnam-notes-smale-spaces-2015}
\leavevmode\vrule height 2pt depth -1.6pt width 23pt, {\em Lecture notes on
  {S}male spaces},  (2015).
\newblock http://www.math.uvic.ca/faculty/putnam.

\bibitem{ruelle-thermodynamic-2nd}
{\sc D.~Ruelle}, {\em Thermodynamic formalism}, Cambridge Mathematical Library,
  Cambridge University Press, Cambridge, second~ed., 2004.
\newblock The mathematical structures of equilibrium statistical mechanics.

\bibitem{schmidt-book}
{\sc K.~Schmidt}, {\em Dynamical systems of algebraic origin}, vol.~128 of
  Progress in Mathematics, Birkh\"auser Verlag, Basel, 1995.

\bibitem{smale-bams}
{\sc S.~Smale}, {\em Differentiable dynamical systems}, Bull. Amer. Math. Soc.,
  73 (1967), pp.~747--817.

\bibitem{weiss-intrinsically-ergodic}
{\sc B.~Weiss}, {\em Intrinsically ergodic systems}, Bull. Amer. Math. Soc., 76
  (1970), pp.~1266--1269.

\bibitem{weiss-sofic-1972}
\leavevmode\vrule height 2pt depth -1.6pt width 23pt, {\em Subshifts of finite
  type and sofic systems}, Monatsh. Math., 77 (1973), pp.~462--474.

\bibitem{wieler_2014}
{\sc S.~Wieler}, {\em Smale spaces via inverse limits}, Ergodic Theory Dynam.
  Systems, 34 (2014), pp.~2066--2092.

\bibitem{williams_1974}
{\sc R.~F. Williams}, {\em Expanding attractors}, Inst. Hautes \'Etudes Sci.
  Publ. Math.,  (1974), pp.~169--203.

\bibitem{yan-conditional-entropy}
{\sc K.~Yan}, {\em Conditional entropy and fiber entropy for amenable group
  actions}, J. Differential Equations, 259 (2015), pp.~3004--3031.

\end{thebibliography}
\end{document}